\documentclass[a4paper,11pt]{article}

\usepackage{amsfonts,amsmath,amsthm}
\usepackage{graphicx}
\usepackage{subcaption}

\newtheorem{theorem}{Theorem}[section]

\newtheorem{lemma}[theorem]{Lemma}

\begin{document}

\title {Harmonic Green and Neumann functions for domains bounded by two intersecting circular arcs}
\author{Hanxing Lin}
\date{}
\maketitle
\begin{center} Math. Institute, FU Berlin, Arnimallee 3, D-14195 Berlin, Germany
\end{center}
\begin{center}
	hxlin@zedat.fu-berlin.de
\end{center}

\begin{flushright}
	{\it Dedicated to Professor Dr. Heinrich Begehr  $\;\;\;\;\;\;\;$ \\
		on the occasion of his $80^{\textrm{th}}$ birthday} $\;\;\;\;\;\;\;$
\end{flushright}

\noindent{\bf Abstract.} The parqueting-reflection principle is shown to also work for constructing harmonic Green functions and harmonic Neumann functions for a class of domains, which are bounded by two arcs in $\mathbb{C}_\infty$ with a special  intersecting angle $\pi/n,\,n\in\mathbb{N}^*$. Applying the Green representation formula and the Neumann representation formula we solve the Dirichlet and Neumann boundary problem to the Poisson equation in these domains.



\section{Circle Reflection in $\mathbb{C}_\infty$}
 Any circle or straight line in $\mathbb{C}$ can be expressed uniformly by an equation of the form
\[
az\bar{z}+\bar{b}z+b\bar{z}+c=0,
\]
where $a,\,c\in\mathbb{R},\,b\in\mathbb{C}$, and $ac-b\bar{b}<0$. It is a straight line if $a=0$, and a circle whose center is $-{b}/{a}$ with a radius ${\sqrt{b\bar{b}-ac}}/{|a|}$ if $a\ne0$. In both cases we just call them  circles in the extended complex plane $\mathbb{C}_\infty$, which is $\mathbb{C}\cup\{\infty\}$.  

We see that a circle in $\mathbb{C}_\infty$ is  determined by a matrix
\[
\left(
\begin{matrix}
a &\bar{b} \\
b & c
\end{matrix}
\right),\ \text{where}\  a,\ c\in\mathbb{R},\ b\in\mathbb{C},\ \mathrm{and} \ ac-b\bar{b}\ne0,
\]
a $2\times2$ Hermitian matrix with negative determinant. We call it a matrix of a circle in $\mathbb{C}_\infty$. A matrix of a circle is unique up to a nonzero real scalar multiple. 
Let $\mathrm{H}^{-}=\{A\in \mathrm{GL}_2(\mathbb{C}):A=A^{*},\,\mathrm{det}(A)<0\}$ be all the $2\times2$ Hermitian matrices with negative determinant. We define an equivalence relation, denoted by $\sim$, in $\mathrm{H}^{-}$: $A\sim B$ in $\mathrm{H}^{-}$ if and only if $A=\lambda B$ for some $0\ne \lambda\in\mathbb{R}$. There exists one to one correspondence between the collection of all circles in $\mathbb{C}_\infty$ and all the equivalent classes in $\mathrm{H}^{-}$. 

Next discussion is about reflections at circles in $\mathbb{C}_\infty$. A reflection at the circle $\{az\bar{z}+\bar{b}z+b\bar{z}+c=0\}$ is defined by the mapping

\[
\begin{array}{cl}
	\phi\,: & \mathbb{C}_\infty\longrightarrow \mathbb{C}_\infty,\\
	& z\mapsto -\frac{b\bar{z}+c}{a\bar{z}+\overline{b}}.
	\end{array}
\]
When the mirror is a straight line, it is just the reflection at the line in typical sense. When the mirror is the circle $\{|z-z_0|=r\}$, it is the circle reflection, which sends $z$ to $z_0+{r^2}/(\bar{z}-\bar{z_0})$.

Since \(\mathbb{C}_\infty\) is the same as  the complex projective line $\mathbb{C}\mathrm{P}^1$, a point in $\mathbb{C}_\infty$ can be represented by a homogeneous coordinate \([z:w]\), where \(z,\,w\in\mathbb{C}\) but cannot both be \(0\). 
It corresponds to \({z}/{w}\) if \(w\neq0\), and to \(\infty\) if \(w=0\). The multiplication of a homogeneous coordinate and an invertible matrix is similarly defined as the normal matrices multiplication, namely as
\[
[z:w]
\left(
\begin{matrix}
a&b\\
c&d
\end{matrix}
\right)
\overset{\mathrm{def}}{=}[az+cw:bz+dw].
\]
The conjugate of a homogeneous coordinate \(\overline{[z:w]}\) can be naturally defined by \([\bar{z}:\bar{w}]\). 
Then a reflection through a circle
\(A=\left(
\begin{smallmatrix}
a&\bar{b}\\
b&c
\end{smallmatrix}
\right)\), denoted by \(R_A\), can be given by
\[
R_A([z:w])=\overline{[z:w]
	\left(
	\begin{matrix}
	a&\bar{b}\\
	b&c
	\end{matrix}\right)
	\left(
	\begin{matrix}
	0 & 1\\
	-1 & 0
	\end{matrix}\right)}
\]

We close this section with a small but useful result. 
\begin{lemma}\label{lma1}
	Reflecting circle \(B\) at  circle \(A\) gives the circle \(AB^{-1}A\).
\end{lemma}
\begin{proof}
Suppose \([z:w]\) is a point on the circle \(B\) and its reflection at the circle \(A\) is \([u:v]\). Because \([z:w]\) satisfies the equality
\[
[z:w]B
\left[
\begin{matrix}
\bar{z}\\
\overline{w}
\end{matrix}
\right]=0,
\]
substituting \([z:w]\) by 
\(
\overline{[u:v]AP}
\), where 
\(
P=	\left(
\begin{smallmatrix}
0 & 1\\
-1 & 0
\end{smallmatrix}\right)
\), 
gives
\[
[u:v]AP\bar{B}\bar{P}^T\bar{A}^T
\left[
\begin{matrix}
\bar{u}\\
\bar{v}
\end{matrix}
\right]=0.
\]
Direct calculation shows that \(P\bar{B}\bar{P}^T\sim B^{-1} \). Therefore \([u:v]\) is on the circle \(AB^{-1}A\).
\end{proof}

\section{Parqueting reflections for domains bounded by two circular arcs in $\mathbb{C}_\infty$}

 Based on the conformal invariant property of harmonic Green function, one way to obtain harmonic Green functions for domains is to construct the conformal mappings from these domains onto a domain whose Green function is clearly known. But unfortunately it is normally not easy to get conformal mappings, and sometime conformal mappings may be too complicated. In these cases the parqueting-reflection principle is employed, which provides a method to construct harmonic Green functions, as well as harmonic Neumann functions. It has been verified working for several special domains whose boundary is composed of circular arcs or straight lines, for instance discs,  disc sectors, half planes, strips, half strips, cones, concentric rings, hyperbolic strips and many more, see [1-8]. 

We aim to apply the Parqueting reflection principle to a special class of domains in the complex plane, which are bounded by two circular arcs that intersect at a special angle $\pi/n$, $ n\in\mathbb{N}$. In this article an arc means specially a segment of a circle or straight line. There exists four cases for this class of domains, as shown in Figure \ref{fig1}.  
\begin{figure}
	\centering
	\includegraphics[width=0.8\linewidth]{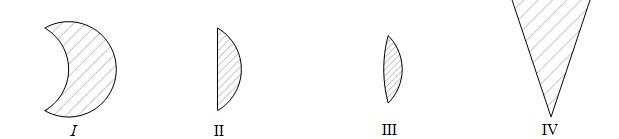}
	\caption{Four cases of domains bounded by two circular arcs}
	\label{fig1}
\end{figure}

Since every disc can be mapped conformally onto the unit disc, without loss of generality, we assume that one arc is a part of the unit circle. Suppose that
\[
D_0=\{z\in\mathbb{C} : z\bar{z}-1\le 0, z\bar{z}\sin(\alpha-\theta)+z\sin{\theta}+\bar{z}\sin{\theta}-\sin(\alpha+\theta)\ge 0 \},
\]
where $\partial{D_0}=C_0\cup C_1$, $C_1$ is a segment of the unit circle. $C_0$ and $C_1$ intersect each other at two points,  $e^{i\alpha}$ and $e^{-i\alpha}$ where $0<\alpha<\pi$, the angle between the two arcs is $\theta$, $0<\theta\le\pi$. Figure \ref{fig2} shows this domain in detail. 

Reflecting $D_0$ at $C_1$ generates a new domain, denoted by $D_2$, $\partial D_2=C_1\cup C_2$. Carrying out successive steps of reflection, each time reflecting the new domain at its new generated boundary, produces a sequence of arcs $C_k$ and domains $D_k$. See that $C_k$ is the reflection of $C_{k-2}$ at $C_{k-1}$, $D_k$ is the reflection of $D_{k-1}$ at $C_k$, and $D_k$ is bounded by $C_{k-1}$ and $C_k$, i.e. $\partial D_k=C_{k-1}\cup C_k$. Since angles are preserved under reflections, $C_{k-1}$ and $C_k$ always intersect at the angle $\theta$. If $\theta=\pi/n$,  after operating $2n-1$ times of reflections we see that $C_{2n}$ coincides with $C_0$, then these circular arcs $C_0, C_1,\cdots,C_{2n-1}$ divide the plane into $2n$ domains, which are just $D_0,D_1,\cdots,D_{2n-1}$. Thereby the sequence of domains $\{D_0,D_1,\cdots,D_{2n-1}\}$ provides a parqueting of the complex plane, i.e. $\mathbb{C}=\bigcup_{k=0}^{2n-1}\overline{D_k}$.

\begin{figure}
	\centering
	\includegraphics[width=0.8\linewidth]{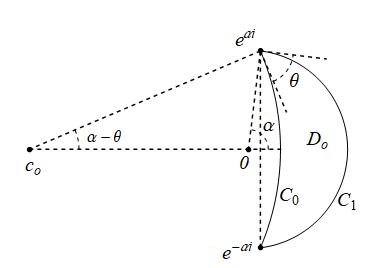}
	\caption{Depiction of $D_0$}
	\label{fig2}
\end{figure}

The following lemma provides the matrices of the boundary arcs.
\begin{lemma}
	The matrix of $C_k$ is
	\[
	M_{C_k}=\left(
	\begin{matrix}
	-\sin(\alpha+(k-1)\theta) & \sin(k-1)\theta  \\
    \sin(k-1)\theta  & \sin(\alpha-(k-1)\theta) 
	
	\end{matrix}
	\right).
	\]
	
\end{lemma}
\begin{proof}
	Verify this conclusion by induction on $k$.
	\[
	M_{C_0}=
	\left(
	\begin{matrix}
	\sin(\alpha-\theta) & \sin{\theta}  \\
	\sin{\theta}  & -\sin(\alpha+\theta) 
	\end{matrix}
	\right)
	\sim
	\left(
	\begin{matrix}
	-\sin(\alpha-\theta) & \sin(-\theta)  \\
	\sin(-\theta)  & \sin(\alpha+\theta) 
	\end{matrix}
	\right),
	\]
	and 
	\[
	M_{C_1}=
	\left(
	\begin{matrix}
	1 & 0  \\
	0 & -1
	\end{matrix}
	\right)
	\sim
	\left(
	\begin{matrix}
	-\sin\alpha & 0  \\
	0 & \sin\alpha 
	\end{matrix}
	\right),
	\]
	they both coincide with the formula. 
	Assume that 
	\[
	M_{C_{k-2}}=\left(
	\begin{matrix}
	-\sin(\alpha+(k-3)\theta) & \sin(k-3)\theta \\
	\sin(k-3)\theta  & \sin(\alpha-(k-3)\theta) 
	\end{matrix}
	\right),
	\]
	and 
	\[
	M_{C_{k-1}}=\left(
	\begin{matrix}
	-\sin(\alpha+(k-2)\theta) & \sin(k-2)\theta  \\
	\sin(k-2)\theta  & \sin(\alpha-(k-2)\theta) 
	\end{matrix}
	\right),
	\]
	since $C_{k}$ is obtained by reflecting $C_{k-2}$ through $C_{k-1}$, via Lemma \ref{lma1} we validate that
	\begin{equation*}
	\begin{split}
	M_{C_k} & =M_{C_{k-1}}M^{-1}_{C_{k-2}}M_{C_{k-1}}\\
	& \sim
	\left(
	\begin{matrix}
	-\sin(\alpha+(k-1)\theta) & \sin(k-1)\theta  \\
	\sin(k-1)\theta  & \sin(\alpha-(k-1)\theta)
	\end{matrix}
	\right)
	\end{split}
	\end{equation*}	
\end{proof}

 Let $z_0=z\in D_0$. The successive reflections which generate the domains $\{D_k\}$ also produce a series of reflection points generated by $z_0$, denoted by $\{z_k\}$. It means that $z_k\in D_k$ and reflecting $z_k$ at the circle $C_{k+1}$ gives $z_{k+1}\in D_{k+1}$. Especially, reflecting $z_0$ at the unit circle $C_1$ gives $z_1=1/\bar{z}\in D_1$. Notice that  $z_k=z_{2n-k-1}$ holds when $z$ belongs to circle $C_0$, while $z_{2k}=z_{2k+1}$ holds when $z$ lies on the circle $C_1$. Actually, there are varieties of viewpoints to investigate the generation of $\{D_k\}$ and $\{z_k\}$. Reflecting $D_0$ at $C_k$ also produces $D_{2k-1}$, while reflecting $D_1$ at $C_{k+1}$ turns out to be $D_{2k}$. It means that  $z_{2k}$ is the image of $z_1$ under reflection through $C_{k+1}$, while $z_{2k+1}$ be the image of $z_0$ under the reflection through $C_{k+1}$. Via the discussion of reflections in Section 1, we deduce that
\begin{equation*}
\begin{split}
z_{2k} & =\overline{[1:\bar{z}]M_{C_{k+1}}\left(
	\begin{matrix}
	0 & 1\\
	-1 & 0
	\end{matrix} \right)}\\
& =\frac{-z\sin(\alpha-k\theta)-\sin k\theta}{z\sin k\theta-\sin(\alpha+k\theta)},
\end{split}
\end{equation*}
and
\begin{equation*}
\begin{split}
z_{2k+1} & =\overline{[z:1]M_{C_{k+1}}\left(
	\begin{matrix}
	0 & 1\\
	-1 & 0
	\end{matrix} \right)}\\
& =\frac{\bar{z}\sin k\theta+\sin(\alpha-k\theta)}{\bar{z}\sin(\alpha+k\theta)-\sin k\theta}.
\end{split}
\end{equation*}

\section{Harmonic Green function of \bf{$D_0$}}

In this section we apply the parqueting-reflection principle to construct the harmonic Green function for $D_0$ and then solve the Dirichlet problem of Poisson equation in $D_0$.   

Consider 
\[
F(z,\zeta)=\prod_{k=0}^{n-1}\frac{\zeta-z_{2k+1}}{\zeta-z_{2k}}=P(z,\zeta)\times\prod_{k=0}^{n-1}\frac{z\sin k\theta-\sin(\alpha+k\theta)}{\bar{z}\sin(\alpha+k\theta)-\sin k\theta},
\]
where
\begin{equation*}
\begin{split}
P(z,\zeta) & =\prod_{k=0}^{n-1} \frac{\bar{z}\zeta\sin(\alpha+k\theta)-(\bar{z}+\zeta)\sin k\theta-\sin(\alpha-k\theta)}{z\zeta \sin k\theta+z\sin(\alpha-k\theta)-\zeta\sin(\alpha+k\theta)+\sin k\theta} \\
& =\frac{\bar{z}\zeta-1}{\zeta-z}\prod_{k=1}^{n-1} \frac{\bar{z}\zeta\sin(\alpha+k\theta)-(\bar{z}+\zeta)\sin k\theta-\sin(\alpha-k\theta)}{ z\zeta\sin k\theta+z\sin(\alpha-k\theta)-\zeta\sin(\alpha+k\theta)+\sin k\theta}.
\end{split}
\end{equation*}
It is obvious that $\log|P(z,\zeta)|^2$ as a function in variable $z$ is harmonic in $D_0\setminus\{\zeta\}$, and $\log|P(z,\zeta)|^2+\log|\zeta-z|^2$ is harmonic in $D_0$. Before showing that $\log|P(z,\zeta)|^2$ vanishes on the boundary $\partial D_0$, firstly we study some properties of $F(z,\zeta)$.

\begin{lemma}\label{lm}
If $\theta=\frac{\pi}{n}$, then 	
\(
\prod_{k=0}^{n-1}\left|\frac{z\sin k\theta-\sin(\alpha+k\theta)}{\bar{z}\sin(\alpha+k\theta)-\sin k\theta}\right|=1
\)
on the boundary of $D_0$.
\end{lemma}
\begin{proof}
	If $z\in C_0$, $z\bar{z}\sin(\alpha-\theta)+z\sin\theta+\bar{z}\sin\theta-\sin(\alpha+\theta)=0 $, then $|z\sin(\alpha-\theta)+\sin\theta|=\sin\alpha$. Substituting $\bar{z}$ by $\frac{-z\sin\theta+\sin(\alpha+\theta)}{z\sin(\alpha-\theta)+\sin\theta}$, we get
\begin{equation*}	
\begin{split}
&\prod_{k=0}^{n-1}\left|\frac{z\sin k\theta-\sin(\alpha+k\theta)}{\bar{z}\sin(\alpha+k\theta)-\sin k\theta}\right|\\
=&\prod_{k=0}^{n-1}\left|\frac{z\sin k\theta-\sin(\alpha+k\theta)}{\frac{-z\sin\theta+\sin(\alpha+\theta)}{z\sin(\alpha-\theta)+\sin\theta}\sin(\alpha+k\theta)-\sin k\theta}\right|\\
=& \prod_{k=0}^{n-1}\left|\frac{z\sin k\theta-\sin(\alpha+k\theta)}{z\sin(k+1)\theta-\sin(\alpha+(k+1)\theta)}\frac{z\sin(\alpha-\theta)+\sin\theta}{\sin\alpha}\right|\\
=&\left|\frac{-\sin\alpha}{z\sin n\theta-\sin(\alpha+n\theta)}\right|\\
=&\left|\frac{-\sin\alpha}{z\sin\pi-\sin(\alpha+\pi)}\right|\\
=&1.
\end{split}
\end{equation*}
In the other case, $z\in C_1$, $z\bar{z}=1$, then
\begin{equation*}
	\prod_{k=0}^{n-1}\left|\frac{z\sin k\theta-\sin(\alpha+k\theta)}{\bar{z}\sin(\alpha+k\theta)-\sin k\theta}\right|=\prod_{k=0}^{n-1}\left|\frac{z\sin k\theta-\sin(\alpha+k\theta)}{\sin(\alpha+k\theta)-z\sin k\theta}\right|=1.
\end{equation*}
The conclusion holds for both cases.
\end{proof}

Noticing that $\lim\limits_{z\to \partial D_0}F(z,\zeta)=1$, it implies immediately $\lim\limits_{z\to\partial D_0}|P(z,\zeta)|=1$ via Lemma \ref{lm}. Then $\log|P(z,\zeta)|^2$ vanishes on the boundary. Therefore we have verified the harmonic Green function of $D_0$,  claimed as follows.
\begin{theorem}
 \(G_1(z,\zeta)=\log|P(z,\zeta)|^2\) is the  harmonic Green function of the domain $D_0$.
	
\end{theorem}

Next we discuss the Poisson kernel for $D_0$. Let $p(z,\zeta)=-\frac{1}{2}\partial_{\nu_\zeta}G_1(z,\zeta)$ be the Poisson kernel for $D_0$, where $z\in D_0$, $\zeta\in\partial D_0$. We calculate $p(z,\zeta)$ directly. 
\begin{equation*}
\begin{split}
\partial_{\zeta}G_1(z,\zeta)=&-\sum_{k=0}^{n-1}\frac{z\sin{k\theta}-\sin(\alpha+k\theta)}{z\zeta\sin{k\theta}+z\sin(\alpha-k\theta)-\zeta\sin(\alpha+k\theta)+\sin{k\theta}}\\
&+\sum_{k=0}^{n-1}\frac{\bar{z}\sin(\alpha+k\theta)-\sin{k\theta}}{\bar{z}\zeta\sin(\alpha+k\theta)-\bar{z}\sin{k\theta}-\zeta\sin{k\theta}-\sin(\alpha-k\theta)}.
\end{split}
\end{equation*}
On the boundary $C_0\subset\{\zeta\bar{\zeta}\sin(\alpha-\theta)+\zeta\sin\theta+\bar{\zeta}\sin\theta-\sin(\alpha+\theta)=0\}$, the outward normal vector $\nu_\zeta$ is
\[
-\frac{\zeta+\frac{\sin\theta}{\sin(\alpha-\theta)}}{\frac{\sin\alpha}{\sin(\alpha-\theta)}},
\]
while the outward normal derivative is
\[
\partial_{\nu_\zeta}=-\left(\frac{\zeta\sin(\alpha-\theta)+\sin\theta}{\sin\alpha}\partial_\zeta+\frac{\bar{\zeta}\sin(\alpha-\theta)+\sin\theta}{\sin\alpha}\partial_{\bar{\zeta}}\right).
\]
(When $\alpha=\theta$, $C_0$ is just a line segment, in this case $\nu_\zeta=-1$ and $\partial_{\nu_\zeta}=-\partial_\zeta-\partial_{\bar{\zeta}}$.)
Obviously
\begin{equation*}
\begin{split}
&\frac{\zeta\sin(\alpha-\theta)+\sin\theta}{\sin\alpha}\frac{z\sin{k\theta}-\sin(\alpha+k\theta)}{z\zeta \sin{k\theta}+z\sin(\alpha-k\theta)-\zeta\sin(\alpha+k\theta)+\sin{k\theta}}\\
=&\frac{\sin(\alpha-\theta)}{\sin\alpha}-\frac{z\sin(\alpha-(k+1)\theta)+\sin{(k+1)\theta}}{z\zeta \sin{k\theta}+z\sin(\alpha-k\theta)-\zeta\sin(\alpha+k\theta)+\sin{k\theta}}.
\end{split}
\end{equation*}
Replacing $\zeta$ by $\frac{-\bar{\zeta}\sin\theta+\sin(\alpha+\theta)}{\overline{\zeta}\sin(\alpha-\theta)+\sin\theta}$ in 
\[
\frac{\bar{z}\sin(\alpha+k\theta)-\sin{k\theta}}{\bar{z}\zeta\sin(\alpha+k\theta)-(\bar{z}+\zeta)\sin{k\theta}-\sin(\alpha-k\theta)}
\]
gives 
\begin{equation*}
\begin{split}
&-\frac{\bar{z}\sin(\alpha+k\theta)-\sin{k\theta}}{\bar{z}\bar{\zeta}\sin{(k+1)\theta}-\bar{z}\sin(\alpha+(k+1)\theta)+\bar{\zeta}\sin(\alpha-(k+1)\theta)+\sin{(k+1)\theta}}\\
&\times\frac{\bar{\zeta}\sin(\alpha-\theta)+\sin\theta}{\sin\alpha},
\end{split}
\end{equation*}
which implies that
\begin{align*}
\begin{split}
&\frac{\zeta\sin(\alpha-\theta)+\sin\theta}{\sin\alpha}\frac{\bar{z}\sin(\alpha+k\theta)-\sin{k\theta}}{\bar{z}\zeta \sin(\alpha+k\theta)-(\bar{z}+\zeta)\sin{k\theta}-\sin(\alpha-k\theta)}\\
=&-\frac{\bar{z}\sin(\alpha+k\theta)-\sin{k\theta}}{\bar{z}\bar{\zeta}\sin{(k+1)\theta}-\bar{z}\sin(\alpha+(k+1)\theta)+\bar{\zeta}\sin(\alpha-(k+1)\theta)+\sin{(k+1)\theta}}\\
&\times\frac{|\zeta\sin(\alpha-\theta)+\sin\theta|^2}{\sin^2\alpha}\\
=&-\frac{\bar{z}\sin(\alpha+k\theta)-\sin{k\theta}}{\bar{z}\bar{\zeta}\sin{(k+1)\theta}-\bar{z}\sin(\alpha+(k+1)\theta)+\bar{\zeta}\sin(\alpha-(k+1)\theta)+\sin{(k+1)\theta}}.
\end{split}
\end{align*}
Let $l:=n-k-1$, since $n\theta=\pi$, we have
\begin{align*}
\begin{split}
&\frac{\zeta\sin(\alpha-\theta)+\sin\theta}{\sin\alpha}\frac{\bar{z}\sin(\alpha+k\theta)-\sin{k\theta}}{\bar{z}\zeta \sin(\alpha+k\theta)-(\bar{z}+\zeta)\sin{k\theta}-\sin(\alpha-k\theta)}\\
=&\frac{\bar{z}\sin(\alpha-(l+1)\theta)+\sin{(l+1)\theta}}{\bar{z}\bar{\zeta}\sin{l\theta}+\bar{z}\sin(\alpha-l\theta)-\bar{\zeta}\sin(\alpha+l\theta)+\sin{l\theta}}.
\end{split}
\end{align*}
Then 
\begin{equation*}
\begin{split}
&\frac{\zeta\sin(\alpha-\theta)+\sin\theta}{\sin\alpha}\partial_{\zeta}G_1(z,\zeta)\\
=&-\frac{n\sin(\alpha-\theta)}{\sin\alpha}+2\sum_{k=0}^{n-1}\mathrm{Re}\left(\frac{z\sin(\alpha-(k+1)\theta)+\sin{(k+1)\theta}}{z\zeta \sin{k\theta}+z\sin(\alpha-k\theta)-\zeta\sin(\alpha+k\theta)+\sin{k\theta}}\right).
\end{split}
\end{equation*}
Thus in the case of $\zeta\in C_0$, $z\in D_0$, the Poisson kernel
\begin{equation*}
\begin{split}
&p(z,\zeta)\\
=&-\frac{1}{2}\partial_{\nu_\zeta}G_1(z,\zeta)\\
=&\mathrm{Re}\left(\frac{\zeta\sin(\alpha-\theta)+\sin\theta}{\sin\alpha}\partial_{\zeta}G_1(z,\zeta)\right)\\
=&-\frac{n\sin(\alpha-\theta)}{\sin\alpha}+2\sum_{k=0}^{n-1}\mathrm{Re}\left(\frac{z\sin(\alpha-(k+1)\theta)+\sin{(k+1)\theta}}{z\zeta \sin{k\theta}+z\sin(\alpha-k\theta)-\zeta\sin(\alpha+k\theta)+\sin{k\theta}}\right).
\end{split}
\end{equation*}
On the other boundary $C_1$, the outward normal derivative is $\partial_{\nu_\zeta}=\zeta\partial_{\zeta}+\bar{\zeta}\partial_{\bar{\zeta}}$.
\begin{equation*}
\begin{split}
&\zeta\partial_{\zeta}G_1(z,\zeta)\\
=&-\sum_{k=0}^{n-1}\frac{z\zeta\sin{k\theta}-\zeta\sin(\alpha+k\theta)}{z\zeta\sin{k\theta}+z\sin(\alpha-k\theta)-\zeta\sin(\alpha+k\theta)+\sin{k\theta}}\\
&+\sum_{k=0}^{n-1}\frac{\bar{z}\zeta\sin(\alpha+k\theta)-\zeta\sin{k\theta}}{\bar{z}\zeta\sin(\alpha+k\theta)-\bar{z}\sin{k\theta}-\zeta\sin{k\theta}-\sin(\alpha-k\theta)}\\
\overset{\zeta\bar{\zeta}=1}{=\joinrel=\joinrel=}&\sum_{k=0}^{n-1}\left(-1+\frac{z\sin(\alpha-k\theta)+\sin{k\theta}}{z\zeta\sin{k\theta}+z\sin(\alpha-k\theta)-\zeta\sin(\alpha+k\theta)+\sin{k\theta}}\right)\\
&+\sum_{k=0}^{n-1}\frac{\bar{z}\sin(\alpha+k\theta)-\sin{k\theta}}{\bar{z}\sin(\alpha+k\theta)-\bar{z}\bar{\zeta}\sin{k\theta}-\sin{k\theta}-\bar{\zeta}\sin(\alpha-k\theta)}\\
\overset{l:=n-k}{=\joinrel=\joinrel=\joinrel=}&-n+\sum_{k=0}^{n-1}\frac{z\sin(\alpha-k\theta)+\sin{k\theta}}{z\zeta\sin{k\theta}+z\sin(\alpha-k\theta)-\zeta\sin(\alpha+k\theta)+\sin{k\theta}}\\
&+\sum_{l=1}^{n}\frac{\bar{z}\sin(\alpha-l\theta)+\sin{l\theta}}{\bar{z}\bar{\zeta}\sin{l\theta}+\bar{z}\sin(\alpha-l\theta)-\bar{\zeta}\sin(\alpha+l\theta)+\sin{l\theta}}\\
=&-n+2\sum_{k=0}^{n-1}\mathrm{Re}\left(\frac{z\sin(\alpha-k\theta)+\sin{k\theta}}{z\zeta\sin{k\theta}+z\sin(\alpha-k\theta)-\zeta\sin(\alpha+k\theta)+\sin{k\theta}}\right).
\end{split}
\end{equation*}
Thus in the case of $\zeta\in C_1$, $z\in D_0$, the Poisson kernel 
\begin{equation*}
\begin{split}
&p(z,\zeta)\\
=&-\mathrm{Re}(\zeta\partial_{\nu_\zeta}G_1(z,\zeta))\\
=&n-2\sum_{k=0}^{n-1}\mathrm{Re}\left(\frac{z\sin(\alpha-k\theta)+\sin{k\theta}}{z\zeta\sin{k\theta}+z\sin(\alpha-k\theta)-\zeta\sin(\alpha+k\theta)+\sin{k\theta}}\right).
\end{split}
\end{equation*}

\noindent{\bf Remark 1. }The definition of $p(z,\zeta)$ can be extended for $z\in\partial D_0$. When $z\in\partial D_0,$  $\partial_{\zeta}G_1(z,\zeta)=0$ because of $G_1(z,\zeta)=0$. Therefore if $z,\zeta\in \partial D_0$ and $z\neq\zeta$, then $p(z,\zeta)=0$.
 
 Further investigation about the boundary behavior of the Poisson kernel is necessary. Let $g_0(z,\zeta)$ be the Green function for the domain \[\{\zeta\bar{\zeta}\sin(\alpha-\theta)+\zeta\sin\theta+\bar{\zeta}\sin\theta-\sin(\alpha+\theta)\ge0\},\]
and $p_0(z,\zeta)$ be the corresponding Poisson kernel. It is easy to check that
 \[
 g_0(z,\zeta)=\log\left|\frac{\bar{z}\zeta\sin(\alpha-\theta)+\bar{z}\sin\theta+\zeta\sin\theta-\sin(\alpha+\theta)}{(z-\zeta)\sin\alpha}\right|^2,
 \]
\[
p_0(z,\zeta)=-\frac{1}{2}\partial_{\nu_\zeta}g_0(z,\zeta)=-\frac{\sin(\alpha-\theta)}{\sin\alpha}+2\,\mathrm{Re}\left(\frac{z\sin(\alpha-\theta)+\sin\theta}{(z-\zeta)\sin\alpha}\right).
\]
Denote $g_1(z,\zeta)$ and $p_1(z,\zeta)$ respectively as the Green function and the Poisson kernel for the unit disc $\{|\zeta|\le1\}$. It is well known that
\[
g_1(z,\zeta)=\log\left|\frac{\bar{z}\zeta-1}{z-\zeta}\right|^2,
\]
\begin{align*}
p_1(z,\zeta)=-\frac{1}{2}\partial_{\nu_\zeta}g_1(z,\zeta)=1-2\,\mathrm{Re}\left(\frac{z}{z-\zeta}\right).
\end{align*}
With these notations, an boundary property of $p_(z,\zeta)$ is shown in next lemma.
\begin{lemma}\label{pklem}
      $\lim\limits_{z\to C_0}p(z,\zeta)=\lim\limits_{z\to C_0}p_0(z,\zeta)$ if $\zeta\in C_0$, while $\lim\limits_{z\to C_1}p(z,\zeta)=\lim\limits_{z\to C_1}p_1(z,\zeta)$ if $\zeta\in C_1$.
\end{lemma}
\begin{proof}
If $z$ goes to $C_0$, replacing $z$ by $\frac{-\bar{z}\sin\theta+\sin(\alpha+\theta)}{\bar{z}\sin(\alpha-\theta)+\sin\theta}$, we have
\begin{equation*}
\begin{split}
&\partial_{\zeta}G_1(z,\zeta)-\partial_{\zeta}g_0(z,\zeta)\\
=&-\sum_{k=1}^{n-1}\frac{\bar{z}\sin(\alpha+(k-1)\theta)-\sin{(k-1)\theta}}{\bar{z}\zeta\sin(\alpha+(k-1)\theta)-\bar{z}\sin{(k-1)\theta}-\zeta\sin{(k-1)\theta}-\sin(\alpha-(k-1)\theta)}\\
&+\sum_{k=0}^{n-2}\frac{\bar{z}\sin(\alpha+k\theta)-\sin{k\theta}}{\bar{z}\zeta\sin(\alpha+k\theta)-\bar{z}\sin{k\theta}-\zeta\sin{k\theta}-\sin(\alpha-k\theta)}\\
=&0.
\end{split}
\end{equation*}
Hence $\partial_{\nu_\zeta}G_1(z,\zeta)-\partial_{\nu_\zeta}g_0(z,\zeta)=0$ when $z$ goes to $C_0$.
Likewise, if $z$ goes to $C_1$,
 substituting $\bar{z}$ by $1/z$, then we have
\[
\begin{split}
&\partial_{\zeta}G_1(z,\zeta)-\partial_{\zeta}g_1(z,\zeta)\\
=&-\sum_{k=1}^{n-1}\frac{z\sin{k\theta}-\sin(\alpha+k\theta)}{z\zeta\sin{k\theta}+z\sin(\alpha-k\theta)-\zeta\sin(\alpha+k\theta)+\sin{k\theta}}\\
&+\sum_{k=1}^{n-1}\frac{\bar{z}\sin(\alpha+k\theta)-\sin{k\theta}}{\bar{z}\zeta\sin(\alpha+k\theta)-\bar{z}\sin{k\theta}-\zeta\sin{k\theta}-\sin(\alpha-k\theta)}\\
=&0.
\end{split}
\]
Hence $\partial_{\nu_\zeta}G_1(z,\zeta)-\partial_{\nu_\zeta}g_1(z,\zeta)=0$ when $z$ goes to $C_1$.
 
By the above discussion, we deduce that, for $\zeta\in C_0$
\[
\lim\limits_{z\to C_0}p(z,\zeta)=\lim\limits_{z\to C_0}p_0(z,\zeta),
\]
while for $\zeta\in C_1$
\[
\lim\limits_{z\to C_1}p(z,\zeta)=\lim\limits_{z\to C_1}p_1(z,\zeta).
\]
\end{proof}
Then the following result is obtained on the basis of  Remark 1 and Lemma \ref{pklem}.
\begin{theorem}\label{pkthm}
	For $\gamma\in \mathrm{C}(\partial D_0; \mathbb{C})$ the boundary behavior 
	\[
	\lim\limits_{z\to \zeta} \frac{1}{2\pi}\int_{\partial D_0}\gamma(\tilde{\zeta})p(z,\tilde{\zeta})\mathrm{d}s_{\tilde\zeta}=\gamma(\zeta)
	\]
	for any $\zeta\in\partial D_0$ holds.
\end{theorem} 

Green representation formula shows that, see \cite{BeBS20},  for any $w(z)\in \mathrm{C}^2(D,\mathbb{C})\cap\mathrm{C}(\bar{D},\mathbb{C})$, 
\[
w(z)=\frac{1}{2\pi}\int_{\partial D_0} w(\zeta)p(z,\zeta)\mathrm{d}s_\zeta-\frac{1}{\pi}\int_{D_0} w_{\zeta\overline{\zeta}}(\zeta)G_1(z,\zeta)\mathrm{d}\xi\mathrm{d}\eta. 
\]
With the properties of Green function and the result of Theorem \ref{pkthm}, this Green representation formula succeeds in giving the  solution of Dirichlet boundary problem for the Poisson equation on the domain $D_0$, as shown in the next theorem.
\begin{theorem}
	The Dirichlet problem $w_{z\bar{z}}=f$ in $D_0$, $w=\gamma$ on $D_0$ for $f\in \mathrm{L}_{p}(D_0;\mathbb{C})$, $p>2$, $\gamma\in \mathrm{C}(\partial D_0,\mathbb{C})$ has a unique solution, which is provided by 
	\[
	w(z)=\frac{1}{2\pi}\int_{\partial D_0} \gamma(\zeta)p(z,\zeta)\mathrm{d}s_\zeta-\frac{1}{\pi}\int_{D_0} f(\zeta)G_1(z,\zeta)\mathrm{d}\xi\mathrm{d}\eta.
	\]
\end{theorem}

\section{Harmonic Neumann function of \bf{$D_0$}}

To construct the harmonic Neumann function for $D_0$, we start with a rational function in variables $z$ and $\zeta$, denoted by $H(z,\zeta)=\prod_{k=0}^{n-1}(\zeta-z_{2k})(\zeta-z_{2k+1})$. Multiplying $H(z,\zeta)$ by the product of all the denominators appearing in the $z_k$ terms, it turns out to be a polynomial function in variables $z$ and $\zeta$, denoted by
\begin{equation*}
\begin{split}
 Q(z,\zeta) =& \prod_{k=0}^{n-1}(\bar{z}\zeta\sin(\alpha+k\theta)-(\bar{z}+\zeta)\sin{k\theta}-\sin(\alpha-k\theta))\\
&\times \prod_{k=0}^{n-1}(z\zeta\sin{k\theta}+z\sin(\alpha-k\theta)-\zeta\sin(\alpha+k\theta)+\sin k\theta). 
\end{split}
\end{equation*}
 $|P(z,\zeta)|$ is symmetric in variables $z$ and $\zeta$ based on the fact that Green functions are symmetric. We also see that $|Q(z,\zeta)|$ is symmetric in $z$ and $\zeta$ by comparing $Q(z,\zeta)$ with $P(z,\zeta)$. So $Q(z,\zeta)$ will be a candidate for creating the harmonic Neumann function of $D_0$.
\begin{lemma}\label{nflem1}
	Let $N_1(z,\zeta)=-\log|Q(z,\zeta)|^2$. Then
	\[
	\partial_{\nu_z}N_1(z,\zeta)=
	\begin{cases}
	\frac{2n\sin(\alpha-\theta)}{\sin\alpha},& \mathrm{for}\ z\in C_0, \zeta\in\overline{D_0}\setminus\{z\};\\
	-2n,&\mathrm{for}\  z\in C_1, \zeta\in\overline{D_0}\setminus\{z\}.
	\end{cases}
	\]
\end{lemma}
\begin{proof} The outward normal derivative is calculated straightforwardly.
\begin{align*}
\begin{split}
\partial_zN_1(z,\zeta)=&-\sum_{k=0}^{n-1}\frac{\zeta\sin{k\theta}+sin(\alpha-k\theta)}{z\zeta\sin{k\theta}+z\sin(\alpha-k\theta)-\zeta\sin(\alpha+k\theta)+\sin{k\theta}}\\
&-\sum_{k=0}^{n-1}\frac{\bar{\zeta}\sin(\alpha+k\theta)-\sin{k\theta}}{z\bar{\zeta}\sin(\alpha+k\theta)-(z+\bar{\zeta})\sin{k\theta}-\sin(\alpha-k\theta)}.
\end{split}
\end{align*}
In the case of $z\in C_1$, $z\bar{z}=1$,

\begin{equation*}
\begin{split}
	& z\partial_zN_1(z,\zeta)\\
	=&-\sum_{k=0}^{n-1}\frac{z\zeta\sin{k\theta}+z\sin(\alpha-k\theta)}{z\zeta \sin{k\theta}+z\sin(\alpha-k\theta)-\zeta\sin(\alpha+k\theta)+\sin{k\theta}}\\
    &-\sum_{k=0}^{n-1}\frac{z\bar{\zeta}\sin(\alpha+k\theta)-z\sin{k\theta}}{z\bar{\zeta}\sin(\alpha+k\theta)-(z+\bar{\zeta})\sin{k\theta}-\sin(\alpha-k\theta)}\\
	=&-\sum_{k=0}^{n-1} \left(1+\frac{\zeta\sin(\alpha+k\theta)-\sin{k\theta}}{z\zeta \sin{k\theta}+z\sin(\alpha-k\theta)-\zeta\sin(\alpha+k\theta)+\sin{k\theta}}\right)\\
    &+\sum_{k=0}^{n-1}\frac{\bar{\zeta}\sin(\alpha+k\theta)-\sin{k\theta}}{\bar{z}\bar{\zeta}\sin(k\theta)+\bar{z}\sin(\alpha-k\theta)-\bar{\zeta}\sin(\alpha+k\theta)+\sin{k\theta}}\\
		=&-n-2i\sum_{k=0}^{n-1}\mathrm{Im}\left(\frac{\zeta\sin(\alpha+k\theta)-\sin{k\theta}}{z\zeta\sin{k\theta}+z\sin(\alpha-k\theta)-\zeta\sin(\alpha+k\theta)+\sin{k\theta}}\right).
		\end{split}
\end{equation*}
Thereby for $z\in C_1$ and $\zeta\in\overline{D_0}\setminus\{z\}$,
\[
\partial_{\nu_z}N_1(z,\zeta)=z\partial_zN_1(z,\zeta)+\bar{z}\partial_{\bar{z}}N_1(z,\zeta)=2\mathrm{Re}(z\partial_zN_1(z,\zeta))=-2n.
\]
In the other case of $z\in C_0$, the outward normal derivative is
\[
\partial_{\nu_z}=-\left(\frac{z\sin(\alpha-\theta)+\sin\theta}{\sin\alpha}\partial_z+\frac{\bar{z}\sin(\alpha-\theta)+\sin\theta}{\sin\alpha}\partial_{\bar{z}}\right).
\]
Directly check that
\begin{equation*}
	\begin{split}
	&\frac{z\sin(\alpha-\theta)+\sin\theta}{\sin\alpha}\frac{\zeta\sin{k\theta}+\sin(\alpha-k\theta)}{z\zeta \sin{k\theta}+z\sin(\alpha-k\theta)-\zeta\sin(\alpha+k\theta)+\sin{k\theta}}\\
	=&\frac{\sin(\alpha-\theta)}{\sin\alpha}+\frac{\zeta\sin(\alpha+(k-1)\theta)-\sin{(k-1)\theta}}{z\zeta \sin{k\theta}+z\sin(\alpha-k\theta)-\zeta\sin(\alpha+k\theta)+\sin{k\theta}},
	\end{split}
\end{equation*}
replacing $z$ by $\frac{-\bar{z}\sin\theta+\sin(\alpha+\theta)}{\bar{z}\sin(\alpha-\theta)+\sin\theta}$ in 
\[
\frac{\bar{\zeta}\sin(\alpha+k\theta)-\sin{k\theta}}{z\bar{\zeta}\sin(\alpha+k\theta)-(z+\bar{\zeta})\sin{k\theta}-\sin(\alpha-k\theta)}
\]
gives 
\begin{equation*}
\begin{split}
	&-\frac{\bar{\zeta}\sin(\alpha+k\theta)-\sin{k\theta}}{\bar{z}\bar{\zeta}\sin{(k+1)\theta}+\bar{z}\sin(\alpha-(k+1)\theta)-\bar{\zeta}\sin(\alpha+(k+1)\theta)+\sin{(k+1)\theta}}\\
&\times\frac{\bar{z}\sin(\alpha-\theta)+\sin\theta}{\sin\alpha}.
\end{split}
\end{equation*}
Then
\begin{equation*}
\begin{split}
	&\frac{z\sin(\alpha-\theta)+\sin\theta}{\sin\alpha}\frac{\bar{\zeta}\sin(\alpha+k\theta)-\sin{k\theta}}{z\bar{\zeta} \sin(\alpha+k\theta)-(z+\bar{\zeta})\sin{k\theta}-\sin(\alpha-k\theta)}\\
	=&-\frac{\bar{\zeta}\sin(\alpha+k\theta)-\sin{k\theta}}{\bar{z}\bar{\zeta} \sin{(k+1)\theta}+\bar{z}\sin(\alpha-(k+1)\theta)-\bar{\zeta}\sin(\alpha+(k+1)\theta)+\sin{(k+1)\theta}}.
\end{split}
\end{equation*}
Hence
\begin{equation*}
	\begin{split}
	&\quad\frac{z\sin(\alpha-\theta)+\sin\theta}{\sin\alpha}\partial_z N_1(z,\zeta)\\
	&=-\frac{n\sin(\alpha-\theta)}{\sin\alpha}-\sum_{k=0}^{n-1}\frac{\zeta\sin(\alpha+(k-1)\theta)-\sin{(k-1)\theta}}{z\zeta \sin{k\theta}+z\sin(\alpha-k\theta)-\zeta\sin(\alpha+k\theta)+\sin{k\theta}}\\
	&\quad+\sum_{k=0}^{n-1}\frac{\bar{\zeta}\sin(\alpha+k\theta)-\sin{k\theta}}{\bar{z}\bar{\zeta} \sin{(k+1)\theta}+\bar{z}\sin(\alpha-(k+1)\theta)-\bar{\zeta}\sin(\alpha+(k+1)\theta)+\sin{(k+1)\theta}}\\
	&=-\frac{n\sin(\alpha-\theta)}{\sin\alpha}-2i\sum_{k=0}^{n-1}\mathrm{Im}\left(\frac{\zeta\sin(\alpha+(k-1)\theta)-\sin{(k-1)\theta}}{z\zeta \sin{k\theta}+z\sin(\alpha-k\theta)-\zeta\sin(\alpha+k\theta)+\sin{k\theta}}\right),
	\end{split}	
\end{equation*}
from which we know that, for $z\in C_0$ and $\zeta\in\overline{D_0}\setminus\{z\}$,
\[
\partial_{\nu_z}N_1(z,\zeta)=-2\,\mathrm{Re}\left(\frac{z\sin(\alpha-\theta)+\sin\theta}{\sin\alpha}\partial_z N_1(z,\zeta)\right)=\frac{2n\sin(\alpha-\theta)}{\sin\alpha}.
\]
\end{proof}
\noindent{\bf Remark 2.} Denote $\sigma(s)=\partial_{\nu_z}N_1(z,\zeta)$ for $z=z(s)\in \partial D_0$ and $\zeta\in\overline{D_0}\setminus\{z\}$, where $s$ is the arc length parameter, then restating the result of Lemma \ref{nflem1} gives
\[
\sigma(s)=
\begin{cases}
\frac{2n\sin(\alpha-\theta)}{\sin\alpha},& \mathrm{for}\ z(s)\in C_0;\\
-2n,&\mathrm{for}\  z(s)\in C_1.
\end{cases}
\]
Moreover 
\begin{align*}
\begin{split}
-\frac{1}{4\pi}\int_{\partial D_0}\sigma(s)\mathrm{d}s=&-\frac{1}{4\pi}\left(\frac{2n\sin(\alpha-\theta)}{\sin\alpha}\int_{C_0}\mathrm{d}s-2n\int_{C_1}\mathrm{d}s
\right)\\
=&-\frac{1}{4\pi}(4n(\alpha-\theta)-4n\alpha)\\
=&1.
\end{split}
\end{align*}

From the construction of $N_1(z,\zeta)$ and Remark 2, it is seen that as a function of variable $z$ for any $\zeta\in \overline{D_0}$
\begin{itemize}
	\item $N_1(z,\zeta)$ is harmonic in $D_0\setminus\{\zeta\}$ and continuously differentiable in $\overline{D_0}\setminus\{\zeta\}$,
	\item $N_1(z,\zeta)+\log|\zeta-z|^2$ is harmonic in $D_0$,
	\item  $\partial_{\nu_z}N_1(z,\zeta)=\sigma(s)$ for $z=z(s)\in \partial D_0$, where $s$ is the arc length parameter. The density function $\sigma$ is a real-valued, piecewise constant function of $s$ with finite mass $\int_{\partial D_0}\sigma(s)\mathrm{d}s.$
\end{itemize}
We call $N_1(z,\zeta)$ a Neumann function for the domain $D_0$. 

{\bf Remark 3.} We conjecture that the term $\int_{\partial D_0}\sigma(z)N_1(z,\zeta)\mathrm{d}s_z$ is constant for $\zeta\in D_0$ . 
If this conjecture can be verified, suppose the constant is $K$,  then modifying $N_1(z,\zeta)$ by $N_1(z,\zeta)-K$ will give an unique Neumann function for domain $D_0$ which satisfies an extra condition
\begin{itemize}
	\item  $\int_{\partial D_0}\sigma(z)N_1(z,\zeta)\mathrm{d}s_z=0$ (normalization condition).
\end{itemize} 

The missing of normalization condition for the Neumann function doesn't affect the  following discussion.

\begin{lemma}\label{nflimit}
	For $\zeta\in C_0$, 
	\[
	\lim\limits_{z\to C_0}\left\{\mathrm{Re}\left(-\frac{z\sin(\alpha-\theta)+\sin\theta}{\sin\alpha}\partial_z N_1(z,\zeta)\right)-p_0(z,\zeta)\right\}=\frac{n\sin(\alpha-\theta)}{\sin\alpha};
	\]
	while for $\zeta\in C_1$, 
	\[
	\lim\limits_{z\to C_1}\{\mathrm{Re}(z\partial_z N_1(z,\zeta))-p_1(z,\zeta)\}=-n.
	\]
\end{lemma}
\begin{proof}
	For $\zeta\in C_0$, 
\begin{equation*}
\begin{split}
&\quad-\frac{z\sin(\alpha-\theta)+\sin\theta}{\sin\alpha}\partial_z N_1(z,\zeta)\\
=&2\frac{z\sin(\alpha-\theta)+\sin\theta}{(z-\zeta)\sin\alpha}+T_0(z,\zeta),
\end{split}	
\end{equation*}
where 
\begin{align*}
\begin{split}
	T_0(z,\zeta)=&\sum_{k=1}^{n-1}\frac{\zeta\sin{k\theta}+sin(\alpha-k\theta)}{z\zeta\sin{k\theta}+z\sin(\alpha-k\theta)-\zeta\sin(\alpha+k\theta)+\sin{k\theta}}\\
	&+\sum_{k=0}^{n-2}\frac{\bar{\zeta}\sin(\alpha+k\theta)-\sin{k\theta}}{z\bar{\zeta}\sin(\alpha+k\theta)-(z+\bar{\zeta})\sin{k\theta}-\sin(\alpha-k\theta)}.
\end{split}
\end{align*}
From the proof of Lemma \ref{nflem1}, we see that $T_0(z,\zeta)$  goes to 
\begin{align*}
\begin{split}
\frac{(n-1)\sin(\alpha-\theta)}{\sin\alpha}-2i\sum_{k=1}^{n-1}\mathrm{Im}\left(\frac{\zeta\sin(\alpha+(k-1)\theta)-\sin{(k-1)\theta}}{z\zeta \sin{k\theta}+z\sin(\alpha-k\theta)-\zeta\sin(\alpha+k\theta)+\sin{k\theta}}\right)
\end{split}
\end{align*} 
when $z$ goes to $C_0$. Hence 
\begin{align*}
\begin{split}
&\lim\limits_{z\to C_0}\mathrm{Re}\left(-\frac{z\sin(\alpha-\theta)+\sin\theta}{\sin\alpha}\partial_z N_1(z,\zeta)\right)\\
=&\lim\limits_{z\to C_0}2\mathrm{Re}\frac{z\sin(\alpha-\theta)+\sin\theta}{(z-\zeta)\sin\alpha}+\frac{(n-1)\sin(\alpha-\theta)}{\sin\alpha}\\
=&\lim\limits_{z\to C_0}p_0(z,\zeta)+\frac{n\sin(\alpha-\theta)}{\sin\alpha}.
\end{split}
\end{align*}

For $\zeta\in C_1$, 
\[
 z\partial_z N_1(z,\zeta)=-\frac{2z}{z-\zeta}-T_1(z,\zeta),
\]
where 
\begin{align*}
\begin{split}
T_1(z,\zeta)=&\sum_{k=1}^{n-1}\frac{\zeta\sin{k\theta}+sin(\alpha-k\theta)}{z\zeta\sin{k\theta}+z\sin(\alpha-k\theta)-\zeta\sin(\alpha+k\theta)+\sin{k\theta}}\\
&+\sum_{k=1}^{n-1}\frac{\bar{\zeta}\sin(\alpha+k\theta)-\sin{k\theta}}{z\bar{\zeta}\sin(\alpha+k\theta)-(z+\bar{\zeta})\sin{k\theta}-\sin(\alpha-k\theta)}.
\end{split}
\end{align*}
When $z$ goes to $C_1$, $T_1(z,\zeta)$ turns out to be 
\[
(n-1)+2i\sum_{k=1}^{n-1}\mathrm{Im}\left(\frac{\zeta\sin(\alpha+k\theta)-\sin{k\theta}}{z\zeta\sin{k\theta}+z\sin(\alpha-k\theta)-\zeta\sin(\alpha+k\theta)+\sin{k\theta}}\right),
\]
which is also shown in the proof of Lemma \ref{nflem1}. Therefore
\begin{align*}
\begin{split}
\lim\limits_{z\to C_1}\mathrm{Re}(z\partial_z N_1(z,\zeta))&
=-\lim\limits_{z\to C_1}2\mathrm{Re}\frac{z}{z-\zeta}-(n-1)\\
&=\lim\limits_{z\to C_1}p_1(z,\zeta)-n.
\end{split}
\end{align*}
\end{proof}

\noindent{\bf Neumann Boundary Problem:} Find a solution to the Poisson equation $w_{z\bar{z}}=f$ in $D_0$, satisfying $\partial_{\nu_z}w=\gamma$ on $\partial D_0$ except for the two corner points.

We recall the Neumann representation formula here, see \cite{BeBS20}. For a regular domain $D$ in $\mathbb{C}$, any function $w(z)\in \mathrm{C}^2(D,\mathbb{C})\cap\mathrm{C}(\bar{D},\mathbb{C})$ can be represented by 
\begin{align*}
\begin{split}
w(z)=&-\frac{1}{4\pi}\int_{\partial D} w(\zeta)\sigma(\zeta)\mathrm{d}s_\zeta+\frac{1}{4\pi}\int_{\partial D} \partial_{\nu_\zeta}w(\zeta)N_1(z,\zeta)\mathrm{d}s_\zeta\\
&-\frac{1}{\pi}\int_{D} w_{\zeta\bar{\zeta}}(\zeta)N_1(z,\zeta)\mathrm{d}\xi\mathrm{d}\eta,
\end{split}
\end{align*}
where $N_1(z,\zeta)$ is the Neumann function of $D$, $\sigma(\zeta)=\partial_{\nu_z}N_1(z,\zeta)$ for $z\in\partial D$.
This formula is applied to provide the solutions to the Neumann boundary problem, shown as follows.

\begin{theorem}
The Neumann boundary problem is solvable if and only if 
\[
\int_{\partial D_0} \gamma(\zeta)\mathrm{d}s_\zeta
=4\int_{D_0} f(\zeta)\mathrm{d}\xi\mathrm{d}\eta,
\]	
all the solutions are of the following form 
\begin{align*}
w(z)=c+\frac{1}{4\pi}\int_{\partial D_0} \gamma(\zeta)N_1(z,\zeta)\mathrm{d}s_\zeta-\frac{1}{\pi}\int_{D_0} f(\zeta)N_1(z,\zeta)\mathrm{d}\xi\mathrm{d}\eta,
\end{align*}
where $c$ is an arbitrary constant in $\mathbb{C}$.
\end{theorem}
\begin{proof}
The Neumann function is a fundamental solution to the Poisson equation, and the boundary integral $\int_{\partial D_0} \gamma(\zeta)N_1(z,\zeta)\mathrm{d}s_\zeta$ is harmonic in $D_0$. It immediately implies that $w_{z\bar{z}}=f$. 
\[
\partial_{\nu_z}w(z)=\frac{1}{4\pi}\int_{\partial D_0}\gamma(\zeta)\partial_{\nu_z}N_1(z,\zeta)\mathrm{d}s_\zeta-\frac{1}{\pi}\int_{D_0} f(\zeta)\partial_{\nu_z}N_1(z,\zeta)\mathrm{d}\xi\mathrm{d}\eta.
\]
On the basis of Lemma \ref{nflem1} and Lemma \ref{nflimit}, if $\zeta\in C_0$, 
\begin{align*}
\lim\limits_{z\to \zeta}\partial_{\nu_z}w(z)=\gamma(\zeta)+\frac{2n\sin(\alpha-\theta)}{\sin\alpha}\left(\frac{1}{4\pi}\int_{\partial D_0}\gamma(\zeta)\mathrm{d}s_\zeta
-\frac{1}{\pi}\int_{D_0} f(\zeta)\mathrm{d}\xi\mathrm{d}\eta\right),
\end{align*}
if $\zeta\in C_1$, 
\begin{align*}
\lim\limits_{z\to \zeta}\partial_{\nu_z}w(z)=\gamma(\zeta)-2n\left(\frac{1}{4\pi}\int_{\partial D_0}\gamma(\zeta)\mathrm{d}s_\zeta
-\frac{1}{\pi}\int_{D_0} f(\zeta)\mathrm{d}\xi\mathrm{d}\eta\right).
\end{align*}
Then $\partial_{\nu_z}w(z)=\gamma(z)$ on $\partial D_0$ if and only if 
\begin{align*}
\int_{\partial D_0} \gamma(\zeta)\mathrm{d}s_\zeta
=4\int_{D_0} f(\zeta)\mathrm{d}\xi\mathrm{d}\eta.
\end{align*}
Hence function $w(z)$ solves the Neumann problem if the solubility condition is satisfied.

Suppose $\phi(z)$ also solves the Neumann problem, then $\phi(z)-w(z)$ is harmonic in $D_0$ and its normal derivative vanishes on $\partial D_0$. It implies that $\phi(z)-w(z)$ must be a constant. 

To sum up, if the Neumann problem is solvable, all the solutions are of the form as
\[
w(z)=c+\frac{1}{4\pi}\int_{\partial D_0} \gamma(\zeta)N_1(z,\zeta)\mathrm{d}s_\zeta-\frac{1}{\pi}\int_{D_0} f(\zeta)N_1(z,\zeta)\mathrm{d}\xi\mathrm{d}\eta.
\]

\end{proof}

\end{document}